\documentclass[11pt]{amsart}
 \usepackage{graphicx}
 \usepackage{amssymb}
 \usepackage{epstopdf}
 \usepackage{color}

 \newtheorem{lemma}{Lemma}[section]

 \newtheorem{theorem}[lemma]{Theorem}

 \def\R{{\relax\ifmmode I\!\!R\else$I\!\!R$\fi}}
 \newcommand{\beqn}{\begin{equation}}
 \newcommand{\eeqn}{\end{equation}}
 
 \newcommand\eref[1]{(\ref{#1})}

 \newcommand{\al}{\alpha}

 \newcommand{\sgn}{\mathrm{sgn}}

 \newcommand{\argmin}{\mathrm{argmin}}

 \newcommand{\ran}{\rangle}
 \newcommand{\lan}{\langle}
 \newcommand{\abs}[1]{\lvert#1\rvert}

 \title{Rescaled Pure Greedy AlgorithmÊ for Convex
 Optimization}
 \author{Zheming Gao, Guergana Petrova}
 \thanks{%
 Ê Ê This research was supported by the Office of
 Naval Research ContractÊ Ê ONR N00014-11-1-0712 and
 byÊ the NSF Grant DMS 1222715}%
 
\begin{document}
 \maketitle
 \begin{abstract}
 We 
 suggest a new greedy strategy for convex optimization in
 Banach spaces and prove
 its convergent rates under a suitable behavior of theÊ
 modulus of uniform 
 smoothness of the objective function.\\
 
 \noindent
{\bf Key Words:} 
Greedy Algorithms,  Convex Optimization, Rates of Convergence.\\
 \end{abstract}

 \section{Introduction}
 \label{Intr}
 
 The main goal in convex optimization is the development and
 analysis ofÊ algorithmsÊ for solving the problem
 \begin{equation}
 \label{Eq}
 Ê \inf_{x\in \Omega} E(x),
 \end{equation}
 where $E$ is aÊ given convex function and
  $\Omega$ is a bounded convex subset of a  Banach space $X$.
 $E$ is called the {\it objective} function and satisfies the
 convexity condition
 $$
 E(\gamma x+\delta y)\leq \gamma E(x)+\delta E(y), \quad
 x,y\in \Omega, \quad \gamma, \delta\geq 0, \quad
 \gamma+\delta=1.
 $$
 While the classical convex optimization deals with objective
 functions $E$ defined on subsets $\Omega$ in $\R^n$Ê
 for moderate values of $n$, see \cite{BV}, some of the 
 new applications require that the dimension $n$ is quite
 large or even $\infty$. The design of algorithms for such
 cases is quite challenging sinceÊ typical
 convergent results involve $n$, and
 therefore deteriorate severely with the growth of $n$. This
 is the so-calledÊ curse of dimensionality. 
 Recently, there has been an increased interest, see
 \cite{Z,Temlyakov1,Temlyakov2,DT1}, Ê in developing 
 greedy based strategies for solving (\ref{Eq}) with provable
 convergence rate depending only on the properties of $E$ and
 not on the
 dimension of the underlying space.Ê These algorithms
 provideÊ approximations $\{E(x_m)\}$,
 $m=1,2,\ldots$Ê to the solution of \eref{Eq}, with 
 $x_m$ being a linear combination of $m$ elements from a
 given dictionary ${\mathcal D}\subset X$. A dictionary  is any set  ${\mathcal D}$ of norm one elements from $  X$ whose span is dense in $X$. An example of
 a dictionary is any Shauder basis for $X$, or  a union of several Shauder bases. The 
 current algorithms pick the initial approximation 
 $E(x_0)$, $x_0=0$,  the set $\Omega$ as  
 \beqn
\nonumber
 \Omega:=\{x\in X \,\,:\,\,E(x)\leq E(0)\},
 \eeqn
 since the global minimum of $E$ is
 attained on that set, and
 generate a sequence of  successive approximations $E_m:=E(x_m)$,
 $m=1,2,\ldots$ recursively, using the dictionary ${\mathcal
 D}$. Some methods, such as the 
 Weak Chebychev Greedy Algorithm, see \cite{Temlyakov1},
 provide at Step $m$ an approximant  $x_m$ to the point $\bar x$ at which $E$ attains its global minimum,
 determined as 
 $$
 x_m:=\argmin_{x\in span\{\varphi_{j_1}, \ldots,
 \varphi_{j_m}\}} E(x),
 $$
 where $\varphi_{j_1}, \ldots, \varphi_{j_m}$ are suitably
 chosen elements from ${\mathcal D}$. Others
 choose $x_m$ as
 $$
 x_m:=\argmin_{\omega, \lambda \in \R} E(\omega
 x_{m-1}+\lambda \varphi_m),
 $$
 or
 $$
 x_m:=\argmin_{\lambda \in [0,1]} E((1-\lambda)
 x_{m-1}+\lambda \varphi_m)
 $$
 for suitably chosen $\varphi_m\in {\mathcal D}$, where
 $x_{m-1}$ is the previously generated point. Convergence
 rates for these algorithms are proved to be of order
 ${\mathcal O}(m^{1-q})$, where $q$ is a parameter related to the
 smoothness of the objective function $E$. Note that the last
 two approaches are more computationally friendly, since they  require
 solving two or one dimensional optimization problems at each
 step. On the other hand, some of these algorithms work only if 
 the minimum of $E$ is
 attained in the convex hull of ${\mathcal D}$, since
 the approximant $x_m$ is derived  as a
 convex combination of $x_{m-1}$ and $\varphi_m$.

  In this
 paper, we introduce a new greedy algorithm based on one
 dimensional optimization at each step, which does not
 require the solution of \eref{Eq} to belong to the convex
 hull of ${\mathcal D}$ and has a 
 rate of convergence ${\mathcal O}(m^{1-q})$. This algorithm is an
 appropriate modification of the recently introduced
 Rescaled Pure Greedy Algorithm ({\bf RPGA}) for approximating functions in 
 Hilbert and Banach spaces, see \cite{P}. We call it {\bf RPGA(co)}.Ê  
 The paper is organized as follows. In Section \S\ref{cond},
 we list several definitions and known results about 
 convex functions. 
In  section \S\ref{rgreedy}, we  present  the {\bf RPGA(co)} and prove its convergence 
 rate. The rest of the paper describes the weak version of this algorithm.

 \section{Preliminaries}
 \label{cond}
 Let usÊ first recall that a function $E$ is Frechet
 differentiable at $x\in \Omega$ if
 Ê there exists a bounded linear functional, denoted by
 $E'(x)\in X^*$,Ê such that
 $$
 \lim_{h \to 0} \frac {|E( x+h)-E( x)-\lan E'(\bar
 x),h\ran|}{\|h\|} =0.
 $$
 Here we  use the notation 
 $\langle F,x\rangle:=F(x)$ to denote the action of the
 functional $F\in X^*$ on the element $x\in X$.

 The followingÊ lemmas are well known and we simply
 state them.
 \begin{lemma}
 \label{OO}
 Let $E$ be a Frechet differentiable function at each point
 in $\Omega$ andÊ convex on $X$. Then, for all $x\in
 \Omega$ and $x'\in X$,
 $$
 \lan E'(x),x-x' \ran \geq E(x)-E(x').
 $$
 \end{lemma}
  \begin{lemma}
 \label{FT}
 LetÊ $E$ be a Frechet differentiable convex function,
 defined on a convex domain $\Omega$.Ê 
 Then $E$ has a global minimum at $\bar x \in \Omega$ if and
 only if $E'(\bar x)=0$.
 \end{lemma}
 Ê \begin{lemma}
 \label{lm0}Ê 
 Let $F$ be a Frechet differentiable function and $x^*$ be such that 
  $x^*=\argmin \{F(x):{x=t\varphi, t\in
 \R}\}$. 
 Then, $\langle F'( x^*),x^*\rangle =0$.
 \end{lemma}

In this paper, we consider objective functions $E$ that satisfy the following two assumptions.
 \begin{itemize}
 \item
 {\bf Condition 0:}
 $E$ has Frechet derivative $E'(x)\in X^*$
at each point in $\Omega:=\{x\in X \,\,:\,\,E(x)\leq E(0)\}$, $\Omega$ is bounded,  andÊ
$$
\|E'(x)\|\leq M_0, \quad x\in \Omega.
$$
 
 \item
 {\bf Uniform Smoothness (US):} There are constants $0\leq
 \al$, $M>0$, and $1<q\leq 2$, such that 
 for all $x$, $x'$ with $\|x-x'\|\leq M$, $x\in \Omega$,
 \begin{equation}
\nonumber
 E(x')-E(x)-\langle E' (x),x'-x\rangle \leq \al \|x'-x\|^q.
 \end{equation}
 \end{itemize}

The {\bf US} condition on $E$ is closely
 related to a condition on the modulus of smoothness of $E$.
 Recall that 
 for aÊ convex function $E: X \to \R$ and a set $S
 \subset X$, the modulus of smoothness of $E$ on $S$ is
 defined by 
 \begin{equation}
 \nonumber
 Ê \rho(E,u) :=\frac 12 \sup_{ x\in
 S,\|y\|=1}\left\{E(x+uy)+E(x-uy)-2E(x)\right\} ,\quad
 u>0,
 \end{equation}
 andÊ the modulus of uniform smoothness of $E$ on $S$ is
 defined by $\rho_1:= \rho_{1}(E,u)$
 \begin{equation}
\nonumber
 \rho_{1}:=\sup_{x\in S, \|y\|=1, \lambda \in
 (0,1)}\left\{\frac{(1-\lambda)E(x-\lambda uy)+\lambda
 E(x+(1-\lambda) uy)-E(x)}{\lambda(1-\lambda)}\right\}.
 \end{equation}
 TheseÊ two moduli of smoothnessÊ are equivalent
 (seeÊ \cite{Za}, page 205), as the following lemma
 states.
 \begin{lemma}
 \label{equimodu}
 Let $E$ be a convex function defined on $X$, $S\subset X$,
 and $\rho(E,\cdot)$ and $\rho_{1}(E,\cdot)$ be its
 modulus of smoothness and modulus of uniform smoothness,
 respectively. Then we have 
 \begin{equation}
 \nonumber
 4\rho( E,\frac u 2) \leq \rho_{1}(E,u) \leq 2
 \rho(E,u).
 \end{equation}
 \end{lemma}
 The next lemma showsÊ the relation between the modulus
 of  smoothness and the {\bf US} condition. The proof
 of the version cited here can be found in \cite{NP}. Because
 of this lemma, the {\bf US} condition and the
 conditionÊ from \cite{Temlyakov1, Temlyakov2, DT1} on
 the modulus of smoothness
 of $E$ are equivalent.
 \begin{lemma}
 \label{uniformsmooth}
 Let $E$ be a convex function defined on a Banach space $X$
 and $E$ be Frechet differentiable on a set $S \subset X$.
 The following statements are equivalent for any $q\in (1,2]$
 and $M>0$.
 \begin{itemize}
 \label{st1}
 \itemÊ There exists $\al> 0$Ê such that
 for any $x \in S$,  $x' \in X$, $\|x-x'\|\leq M$, 
 \begin{equation}
 \label{unismooth1}
 E(x')-E(x)-\lan E'(x),x'-x \ran \leq \al \|x'-x\|^{q}.
 \end{equation}
 \item 
 \label{st2}
 There exists $\al_{1} >0$,Ê such that
 \begin{equation}
 \label{unismooth2}
 \rho(E,u,S) \leq \al_{1} u^{q}, \quad 0<u\leq M.
 \end{equation}
 \end{itemize}
 \end{lemma}
 \noindent

 Next, we introduce some notation. Let  $\bar x$  be the solution to \eref{Eq}.
We  denote by $\|\bar x\|_1$ its semi-norm with respect to the dictionary ${\mathcal D}$, namely
 $$
 \|\bar x\|_1:=\inf \left\{\sum_{\varphi \in {\mathcal D}} |c_\varphi(\bar
 x)|:\,\,\bar x=\sum_{\varphi \in {\mathcal D}} c_\varphi(\bar
 x)\varphi\right\},
 $$ 
 where the infimum is taken over all possible representations of $\bar x$ as a linear combination 
 of dictionary elements.
Clearly,  the point $\bar x$ at which $E$ attains its global minimum
belongs to the set
$$
\Omega:=\{x:\,E(x)\leq E(0)\},
$$
and in what follows we will consider the minimization problem \eref{Eq} 
over this set. Note that this is a convex set as a level set of a convex function.

 Further in the paper  we will use the followingÊ lemma,
 proved in \cite{NP}.   Other versions of this lemma have been proved in \cite{Tbook}.
  \begin{lemma}
 \label{lmseq}
 Let $\ell>0$, $r>0$, $B>0$, andÊ
 $\{a_m\}_{m=1}^{\infty}$ andÊ $\{r_m\}_{m=2}^{\infty}$
 be sequences of non-negative numbers satisfying the
 inequalities
 $$ a_1\leq B, \quad a_{m+1} \leq a_m(1-
 \frac{r_{m+1}}{r}a_m^\ell), \quad m=1,2,\dots.$$
 Then, we have 
 \begin{equation}
 \label{tuti12}
 a_m \leq
 \max\{1,\ell^{-1/\ell}\}r^{1/\ell}(rB^{-\ell}+\Sigma_{k=2}^m
 r_k)^{-1/\ell}, \quad m=2,3, \ldots.
 \end{equation}
 \end{lemma}

 \section{The Rescaled Pure Greedy Algorithm for Convex Optimization}
 \label{rgreedy}
 In this section, weÊ describe our new algorithm with parameter $\mu$ and dictionary ${\mathcal D}$.
 
 \noindent
  {\bf RPGA(co)($\mu,{\mathcal D}$)}:
 \begin{itemize}
 \itemÊ {\bf Step $0$}: 
 Define $x_0=0$.Ê 
 If $E'(x_0)=0$, stop the algorithm and define $x_k:=x_0=\bar x$,
 $k\geq 1$.
 
 \item {\bf Step $m$}:Ê Assuming $x_{m-1}$ has been
 defined and $E'(x_{m-1})\neq 0$. Choose a direction 
 $\varphi_{j_m}\in {\mathcal D}$ such that 
 $$
 |\langle E'(x_{m-1}),\varphi_{j_m} \rangle|
 =\sup_{\varphi \in \mathcal D} |\langle
 E'(x_{m-1}),\varphi\rangle|.
 $$
 With
 $\hat x_m:=x_{m-1}-\lambda_m\varphi_{j_m}$, 
 where
$$
 \lambda_m:=\sgn\{\lan E'(x_{m-1}), \varphi_{j_m}\ran\}
 \left(\alpha\mu\right)^{-\frac{1}{q-1}}\abs{\lan
 E'(x_{m-1}),\varphi_{j_m}\ran}^{\frac{1}{q-1}}, 
 $$
 $$
 t_m:=\argmin_{t\in \R} E(t\hat x_m), 
 $$
 define the next point to be
 $$
 x_m=t_m\hat x_m.
 $$
 \item If $E'(x_{m})=0$, stop the algorithm and define $x_k=x_m=\bar x$, for $k>m$.
 \item If $E'(x_{m})\neq0$, proceed to Step $m+1$.
 \end{itemize}
Let us observe that , because of Lemma \ref{FT},Ê if
 $E'(x_m)=0$ at Step $m$, the output $x_m$ of the
 algorithm isÊ the minimizer $\bar x$. Note that the
 algorithm requires a  minimization of  the objective function along the
 one dimensional space $span \{\hat x_m\}$. This univariate optimization problem 
 is called line search and  is well studied in optimization theory, see \cite{N}.
 If at Step $m$ we were to use $\hat x_m$ as next approximant and not $x_m$, which is the 
 minimizer of $E$ along the line generated by $\hat x_m$, then the 
 algorithm would be very similar to the {\bf EGA(${\mathcal C})$} from  \cite{Temlyakov1} . The author
 proves  a  convergence rate of ${\mathcal O}(m^{-r})$, for any $r\in (0,\frac{q-1}{q+1})$ for this algorithm under suitable conditions on the 
 parameters. Note that our algorithm, which simply adds a one dimensional optimization at each step,
  makes it possible to achieve an optimal  convergence rate of 
 ${\mathcal O}(m^{1-q})$. Observe also that, in contrast to the other
 greedy algorithms from \cite{Temlyakov1} that rely on 
 one dimensional minimization at each step, this algorithm
 provides convergent results for all $\bar x$, and not only
 for $\bar x$Ê in the convex hull of the dictionary
 $\mathcal D$.ÊÊÊ

 Notice that  all outputs $\{x_k\}_{k=1}^{\infty}$ generated by theÊ
 {\bf RPGA(co)($\mu, {\mathcal D}$) }  are in $\Omega$, sinceÊ $E(x_k)\leq E(0)$. The
 following theorem is our main convergence result.
 \begin{theorem}
 \label{wdga}
 Let theÊ convex  function $E$ satisfy {\bf Condition
 0}Ê and the {\bf US} condition. Then, at  Step $k$, the  {\bf RPGA(co)}($\mu, {\mathcal D}$) with parameter $\mu>\max\{1,\alpha^{-1}M_0M^{1-q}\}$,Ê applied to
 $E$ and 
 aÊ dictionary ${\mathcal D}=\{\varphi\}$ outputs  the point $x_k$,
 where 
 $$
 e_k:=E(x_k)-E(\bar x)\leq C_1 k^{1-q},\quad k\geq 2,
 $$
with
 $C_1=C_1(q,\alpha,E,\mu)$.
 \end{theorem}
 \begin{proof}
  Clearly, we haveÊ $e_1=E(x_1)-E(\bar x)\leq E(0)-E(\bar
 x)$.
 Next, 
 we considerÊ Step $k$, $k=2,3,\ldots$ of the algorithm. Notice that 
 $x_k\in\Omega$, since $E(x_k)\leq E(0)$. 
 The definition  of $\lambda_k$ and the choice of parameter $\mu$ assures that 
$$
\|(x_{k-1}-\lambda_k\varphi_{j_k})-x_{k-1}\|=\left(\frac{\abs{\lan
 E'(x_{k-1}),\varphi_{j_k}\ran}}{\alpha \mu}\right )^{\frac{1}{q-1}}\leq M,
 $$
 and therefore, applying  the {\bf US}Ê  condition  to $(x_{k-1}-\lambda_k \varphi_{j_k})$
 and $x_{k-1}$ gives
\begin{eqnarray}
\nonumber
 E(\hat x_k)&=&E(x_{k-1}-\lambda_k\varphi_{j_k})\leq
 E(x_{k-1})-\lambda_k\lan E'(x_{k-1}),\varphi_{j_k}\ran+\alpha|\lambda_k|^q
 \nonumber \\
 &=&E(x_{k-1})-
 \frac{\mu-1}{\mu}\left(\alpha \mu\right)^{-\frac{1}{q-1}}\abs{\lan
 E'(x_{k-1}),\varphi_{j_k}\ran}^{q/(q-1)},
 \nonumber
 \end{eqnarray}
 where we use the fact that $\|\varphi_{j_k}\|\leq1$.
 Since 
 $E(x_k)\leq E(\hat x_k)$, we derive that 
 \begin{equation}
 \label{ol}
 \,E(x_k)\leq 
 E(x_{k-1})-\frac{\mu-1}{\mu}\left(\alpha
 \mu\right)^{-\frac{1}{q-1}}\abs{\lan
 E'(x_{k-1}),\varphi_{j_k}\ran}^{\frac{q}{q-1}}.
 \end{equation}
 Next, we provide a lower bound for $\abs{\lan
 E'(x_{k-1}),\varphi_{j_k}\ran}$.
 Let us fix $\varepsilon >0$ and choose a representation
 for $\bar x=\sum_{\varphi\in {\mathcal D}}c_\varphi^\varepsilon \varphi$,
 such that 
 $$
 \sum_{\varphi\in {\mathcal D}}|c_\varphi^\varepsilon|<\|\bar x\|_1+\varepsilon.
 $$
 Since $\lan E'(x_{k-1}),x_{k-1}\ran=0$, because of the
 choice of $x_{k-1}$ and Lemma \ref{lm0}, we have that 
 \begin{eqnarray}
 \nonumber
 \lan E'(x_{k-1}),x_{k-1}-\bar x\ran&=&-\lan
 E'(x_{k-1}),\bar
 x\ran=-\sum_{\varphi}c_\varphi^\varepsilon\lan
 E'(x_{k-1}),\varphi\ran\\
 \nonumber
 &\leq& |\lan
 E'(x_{k-1}),\varphi_{j_k}\ran|\sum_{\varphi}|c_\varphi^\varepsilon|
 \nonumber \\
 &<&|\lan
 E'(x_{k-1}),\varphi_{j_k}\ran(\|\bar x\|_1+\varepsilon),
 \nonumber
 \end{eqnarray}
 where we have used the choice of $\varphi_{j_k}$.
We let  $\varepsilon \rightarrow 0$ and  obtain the inequality
 \begin{equation}
 \label{pop}
 \lan E'(x_{k-1}),x_{k-1}-\bar x\ran\leq |\lan
 E'(x_{k-1}),\varphi_{j_k}\ran|\|\bar x\|_1.
 \end{equation}
 On the other hand, Lemma \ref{OO} and \eref{pop} give that
 $$
 \|\bar x\|_1^{-1}e_{k-1}\leq |\lan
 E'(x_{k-1}),\varphi_{j_k}\ran|,
 $$
 which is the desired estimate from below for $\abs{\lan
 E'(x_{k-1}),\varphi_{j_k}\ran}$
 We substitute the latter inequality in \eref{ol} and derive
 $$
 E(x_k)\leq
 E(x_{k-1})-\frac{\mu-1}{\mu}\left(\alpha
 \mu\right)^{-\frac{1}{q-1}}\|\bar x\|_1^{-\frac{q}{q-1}}e_{k-1}^{\frac{q}{q-1}}.
 $$
 Subtracting $E(\bar x)$ from both sides gives
 $$
 e_k\leq e_{k-1}\left (1-\frac{\mu-1}{\mu}\left(\alpha
 \mu\right)^{-\frac{1}{q-1}}\|\bar x\|_1^{-\frac{q}{q-1}}e_{k-1}^{\frac{1}{q-1}}\right )
 $$
Now we apply    Lemma \ref{lmseq} for  the sequence of errors
 $\{e_k\}_{k=1}^\infty$ and
 $$
r_k=\frac{\mu-1}{\mu}, \quad \ell=\frac{1}{q-1}>0, \quad B=E(0)-E(\bar x), \quad
 r=\left(\alphaÊ\mu \|\bar x\|_1^q\right)^{\frac{1}{q-1}},
 $$
and derive that
 $$
 e_k\leq \alphaÊ\mu \|\bar x\|_1^q\left (
 \left (\frac{\alpha\mu\|\bar x\|_1^q}{E(0)-E(\bar x)}\right )^{\frac{1}{q-1}}+\frac{\mu-1}{\mu}(m-1)\right )^{1-q},
 $$
 and the proof is completed.
 \end{proof}
Notice that   we can  optimize with respect to the parameter $\mu$ and select  a specific value for $\mu>\max\{1,\alpha^{-1}M_0M^{1-q}\}$ that will guarantee the 
 best convergence rate in terms of best constants.

 \section{The Weak Rescaled Pure Greedy Algorithm for Convex Optimization}
 \label{greedy}
 In this section, weÊ describe the weak version of our algorithm with weakness 
 sequence $\{\ell_k\}$, $\ell_k\in (0,1]$ $k=1,2,\ldots$, and parameter sequence $\{\mu_k\}$, $\mu_k>\max\{1,\alpha^{-1}M_0M^{1-q}\}$, $k=1,2,\ldots$.
 In the case when $\ell_k=1$ and $\mu_k=\mu$, $k=1,2,\ldots$, the {\bf WRPGA(co)($\{\ell_k\},\{\mu_k\},{\mathcal D}$)} is the {\bf RPGA(co)($\mu,{\mathcal D}$)}.
 The weakness sequence allows us to have some freedom in the selection of the next direction $\varphi_{j_k}$, while the 
 parameter sequence $\{\mu_k\}$ gives more choices in how much to advance along the selected direction $\varphi_{j_k}$.
 
 \noindent
  {\bf WRPGA(co)($\{\ell_k\},\{\mu_k\},{\mathcal D}$)}:
 \begin{itemize}
 \itemÊ {\bf Step $0$}: 
 Define $x_0=0$.Ê 
 If $E'(x_0)=0$, stop the algorithm and define $x_k:=x_0=\bar x$,
 $k\geq 1$.
 
 \item {\bf Step $m$}:Ê Assuming $x_{m-1}$ has been
 defined and $E'(x_{m-1})\neq 0$. Choose a direction 
 $\varphi_{j_m}\in {\mathcal D}$ such that 
 $$
 |\langle E'(x_{m-1}),\varphi_{j_m} \rangle|
 \geq \ell_m\sup_{\varphi \in \mathcal D} |\langle
 E'(x_{m-1}),\varphi\rangle|.
 $$
 With
 $\hat x_m:=x_{m-1}-\lambda_m\varphi_{j_m}$, 
 where
$$
 \lambda_m:=\sgn\{\lan E'(x_{m-1}), \varphi_{j_m}\ran\}
 \left(\alpha\mu_m\right)^{-\frac{1}{q-1}}\abs{\lan
 E'(x_{m-1}),\varphi_{j_m}\ran}^{\frac{1}{q-1}}, 
 $$
 $$
 t_m:=\argmin_{t\in \R} E(t\hat x_m), 
 $$
 define the next point to be
 $$
 x_m=t_m\hat x_m.
 $$
 \item If $E'(x_{m})=0$, stop the algorithm and define $x_k=x_m=\bar x$, for $k>m$.
 \item If $E'(x_{m})\neq0$, proceed to Step $m+1$.
 \end{itemize}
The next theorem is the main result about the convergence
 rate of the {\bf WRPGA(co)}($\{\ell_k\}, \{\mu_k\}, \mathcal D$).
 \begin{theorem}
 \label{wdga}
 Let theÊ convex  function $E$ satisfy {\bf Condition
 0}Ê and the {\bf US} condition. Then, at  Step $k$, the  {\bf WRPGA(co)}($\{\ell_k\}, \{\mu_k\}, \mathcal D$),Ê applied to
 $E$ and 
 aÊ dictionary ${\mathcal D}=\{\varphi\}$ outputs  the point $x_k$,
 where 
 $$
 e_k:=E(x_k)-e(\bar x)\leq \alphaÊ \|x\|_1^q\left
 (C_1+\sum_{j=2}^k(\mu_j-1)\left (\frac{\ell_j}{\mu_j}\right
 )^{\frac{q}{q-1}}\right )^{1-q},\quad k\geq 2,
 $$
 with
 $C_1=C_1(q,\alpha,E)$.
 \end{theorem}
 \begin{proof}
 Similarly to the proof of Theorem \ref{wdga}, we have that 
 $e_1\leq E(0)-E(\bar x)$ and for $k\geq 2$, 
  \begin{equation}
 \label{wol}
 \,E(x_k)\leq 
 E(x_{k-1})-\frac{\mu_k-1}{\mu_k}\left(\alpha
 \mu_k\right)^{-\frac{1}{q-1}}\abs{\lan
 E'(x_{k-1}),\varphi_{j_k}\ran}^{\frac{q}{q-1}}.
 \end{equation}
The same way one can easily derive that
 $$
 \|\bar x\|_1^{-1}\ell_ke_{k-1}\leq |\lan
 E'(x_{k-1}),\varphi_{j_k}\ran|,
 $$
 and thus the estimate
  $$
 e_k\leq e_{k-1}\left (1-\frac{\mu_k-1}{\mu_k}\left(\alpha
 \mu_k\right)^{-\frac{1}{q-1}}\ell_k^{\frac{q}{q-1}}\|\bar x\|_1^{-\frac{q}{q-1}}e_{k-1}^{\frac{1}{q-1}}\right ).
 $$
Now we apply    Lemma \ref{lmseq} for  the sequence of errors
 $\{e_k\}_{k=1}^\infty$ and
 $$
r_k=(\mu_k-1)\left (\frac{\ell_k}{\mu_k}\right)^{\frac{q}{q-1}}, \quad \ell=\frac{1}{q-1}>0, \quad B=E(0)-E(\bar x), \quad
 r=\left(\alpha \|\bar x\|_1^q\right)^{\frac{1}{q-1}},
 $$
and derive that
 $$
 e_k\leq \alphaÊ\|\bar x\|_1^q\left (
 \left (\frac{\alpha\|\bar x\|_1^q}{E(0)-E(\bar x)}\right )^{\frac{1}{q-1}}+\sum_{j=2}^k(\mu_j-1)\left (\frac{\ell_j}{\mu_j}\right)^{\frac{q}{q-1}}\right )^{1-q},
 $$
 and the proof is completed.
  \end{proof}

 \vskip .1in
 Ê 
 Ê \noindent
 Guergana Petrova\\
 Department of Mathematics, Texas A\&M University,
 College Station, TX 77843, USA\\
 Ê gpetrova@math.tamu.edu\\
 \noindent
 Zheming Gao\\
 Department of Mathematics, Texas A\&M University,
 College Station, TX 77843, USA\\
 Êgorgeous1992@tamu.edu\\
 \end{document}